\documentclass[12pt]{amsart}
\usepackage{amsmath,amssymb,amscd}
\usepackage[mathscr]{eucal}
\newcommand{\Spec}{{\mathrm{Spec}\, }}
\newcommand{\D}{\mathbb D}

\newcommand{\cA}{{\mathcal A}}

\newcommand{\cC}{{\mathscr C}}

\newcommand{\sX}{{\mathscr X}}

\newcommand{\cE}{{\mathcal E}}

\newcommand{\cN}{{\mathcal N}}
\newcommand{\cO}{{\mathcal O}}
\newcommand{\cP}{{\mathscr P}}
\newcommand{\cR}{{\mathscr R}}

\newcommand{\la}{\lambda}
\newcommand{\La}{\Lambda}

\newcommand{\sg}{\mathrm{Sing}}

\theoremstyle{plain}
\newtheorem{thm}{Theorem}[section]

\newtheorem{pro}[thm]{Proposition}
\newtheorem{cor}[thm]{Corollary}
\theoremstyle{definition}

\newtheorem{rem}[thm]{Remark}

\newtheorem{voi}[thm]{}

\begin{document}

\title{On the Nash problem for surfaces in positive characteristic}
\author{Augusto Nobile}

\address{Louisiana State University \\
Department of Mathematics \\
Baton Rouge, LA 70803, USA}



\date{December 1, 2018}
\email{nobile@math.lsu.edu}
\begin{abstract}

This paper seeks to prove the bijectivity of the ``Nash mapping'' from the set of irreducible components of the scheme parametrizing analytic arcs on an algebraic surface $X$ whose origin is a singular point, into the set of irreducible components of the exceptional locus of a minimal desingularization $X'$ of $X$ when the base field has positive characteristic. The idea is to view the surface as a specialization of another defined over a field of characteristic zero.  A number of related results are proved. Among them, the construction of a scheme of arcs for a suitable one parameter family of surfaces which, by using a theorem of M. Artin on lifting of normal surface singularities to characteristic zero, seems a reasonable candidate to be the desired tool. But there are some points, necessary for a complete proof, which are not verified yet.

\end{abstract}

\maketitle

\section*{Introduction}

In order to study the singularities of an algebraic variety $X$ over a field $k$, in the mid 1960's John Nash considered analytic arcs on $X$ with origin in $S$, the singular set of the variety. Actually, in the original paper \cite{NA} he used {\it truncations} of such arcs, up to level $n$ (that is morphisms $\Spec (k[[t]]/(t^{n+1})) \to X$), and worked over the complex numbers. He obtained varieties $ X_n^{(S)}$ parametrizing such truncations (see \cite{NA} or \cite{NO}). Later it became customary to use directly $S$-$arcs$, i.e., morphisms $\Spec (K[[t]]) \to X$ sending the closed point to $S$, where $K$ is a field containing $k$. This approch simplifies the presentation, although the arc scheme $X_{\infty}^{(S)}$ parametrizing such arcs is not of finite type over $k$. In this paper we work mostly with actual arcs, rather than truncated ones.

 In general one has to distinguish between the {\it good} irreducible components of $X_{\infty}^{(S)}$, that is those that contain a point corresponding to an arc whose image is not completely included in $S$, and the other components. But in characteristic zero or when $\dim X \le  2$ all components of  $X_{\infty}^{(S)}$ are good. 

Following Nash, if $f:X' \to X$ is a resolution of singularities of $X$, one has an injective map $\cN$ from the set of good 
irreducible components of $X_{\infty}^{(S)}$ to the set of irreducible components of the exceptional locus $E(f)$ of $f$. Now the image of $\cN$ is contained in the set of {\it essential components} of $E(f)$, i.e., those that ``appear'' in any resolution of $X$; see \cite{IK}. For a normal surface $X$, the essential components of any resolution naturally correspond to the irreducible components of a minimal resolution  of $X$. 

In \cite{NA} Nash asks whether the Nash map $\cN$ is bijective.

Ishii and Kollar in \cite{IK} show that the answer in negative if $\dim X > 3$, even in characteristic zero. A similar result was obtained when $\dim X=3$ by De Fernex in \cite{D3}.

On the positive side, Bobadilla and Pereira in \cite{FP} prove that, working over 
 a field of characteristic zero, the Nash map is bijective for surfaces. They 
 show that we may assume that $k=\mathbb C$ and use topological methods. Later this was done  algebraically by De Fernex and Docampo in \cite{DD}, where also some interesting results in higher dimension are obtained.

But, to my knowledge, it is not known what happens with surfaces in positive characteristic. The purpose of this paper is to describe an approach that might yield a proof of the bijectivity of the Nash map $\cN$ for surfaces defined over a field of characteristic $p>0$. 

 As explained in \cite{FP}, we may assume that the base field 
$k$ is algebraically closed and that the surface 
$X$ has a single normal singularity $P$. 
Actually, from the results of \cite{DD} or \cite{FP}, in this case, independently of the characteristic of $k$, to prove the bijectivity of $\cN$ reduces to a simple counting. 
 Let's review this ``classical'' counting process. 

Consider a surface $X$ over  a field $k$ (algebraically closed, of any characteristic),  
with a single normal singularity $P$. Take a minimal resolution  $f:X' \to X$ of our surface. Assume $Ex(f)$ (the exceptional set of $f$) has irreducible components $E_1, \ldots, E_m$. By Zariski's Main Theorem (\cite{H}, Corollary 11.4), since  $X$ is normal and has a single singularity, the desingularization $X'$ is necessarily {\it divisorial}, i.e., each irreducible component of $E(f)$ is one-dimensional; moreover $E(f)$ is connected.
 We write 
$N_r(X,P) := m$, or simply $N_r(X)=m$, since $P$ is the only singularity of $X$.

Consider the Nash space of arcs on $X$ based at $P$, $X_{\infty}^{(P)}$. By the existence of the injective Nash mapping $\cN$, $X_{\infty}^{(P)}$ has finitely many irreducible components, say $C_1, \ldots, C_s$.  The number $s$ (necessarily $\le m$) will be denoted by $N_n(X,P)$, or simply by $N_n(X)$. So, 
 $N_n(X) \le N_r(X)$.

Hence, to prove that the Nash mapping is surjective, it suffices (or is equivalent) to show: $m=s$, i.e., 
$$(1) \qquad N_n(X) = N_r(X) \, .$$
This equality is proved, in characteristic zero,  in \cite{FP} or \cite{DD}.

Our approach is to ``lift'' to characteristic zero, to see that the equality (1) is always valid. The main idea is, given our normal surface $X$ with a single normal singular point $P$, defined over an algebraically closed field $k$  of arbitrary characteristic, to find a family  of surfaces, parametrized by the spectrum
 of a discrete valuation ring $\Lambda$, with residue field $k$ and field of fractions $F$ of characteristic zero, such that the closed fiber is $X$. This may be regarded as a deformation of $X$. Since the geometric general fiber is defined over a field of characteristic zero, the result of \cite{FP} holds, in particular the equality (1) is valid. Hopefully, if the family is good enough, this will imply that the analog of (1) is also valid for the closed fiber $X$, i.e., that the Nash conjecture is true for it. We do not know whether this can be carried out but we hope that a slightly different statement involving instead algebroid surfaces, sufficient to prove the surjectivity of Nash's map, can be obtained. By an algebroid surface we mean the spectrum of a complete noetherian two dimensional local $k$-algebra, with residue field $k$, see \ref{Cal}.

 We shall discuss our work in this direction in the remainder of this article, which comprises three sections. 

In Section \ref{S:C}, we explain how the usual construction of the scheme $X_{\infty}^{(P)}$ of arcs on a variety $X$, relative to a unique singular point $P \in X$, can be extended to suitable families $p:\sX \to U$, where $\sX$ and $U$ are noetherian schemes, equipped with a section $s:U \to \sX$ of $p$, to yield an $\sX$-scheme $\sX_{\infty}^{(s)}$. The data $(\sX,s,U)$, where $\sX$ is a $U$-scheme and $s$ a section as above, will be called a {\it triple}. For simplicity, and because it is the situation interesting to us, we assume $\sX$ and $U$ affine. We study functorial properties of such a construction. With them, we investigate fibers of the morphism $p$ and their arc schemes.

In the ``classical'' case where $X$ is an algebraic variety, the scheme 
$X_{\infty}^{(P)}$ has finitely many irreducible components. We prove that if our triple $(\sX,s,U)$ satisfies a seemingly mild condition, that we call {\it Condition (NO}), then also  $\sX_{\infty}^{(s)}$ has a finite number of irreducible components.

In Section \ref{S:S}, we recall a result of M. Artin (in \cite{A}) that would be the basis of the intended lifting approach. Namely, 
given a normal algebroid surface $X$ over an algebraically closed field of positive characteristic and a resolution $X' \to X$ of its singularities, we would like to find a family parametrized by $U=\Spec (\La)$, where      $\La$ is a complete discrete valuation ring with residue field $k$ and field of fractions of characteristic zero, such that the resolution ``spreads'' over $U$ in a reasonable way. Artin's theorem says that this is possible, perhaps not for the given surface itself, but for a surface $X_0$ closely related to $X$ (in fact $X$ is the normalization of $X_0$), provided the resolution is ``good'', i.e., essentially that the exceptional set is a normal crossings divisor, see \ref{Se:1} and \ref{T:A}. Such a family will be called an {\it Artin family}.

Next we study how the quantity $N_r$ behaves along an Artin family. We check that $N_r$ will be constant if certain intersection numbers (involving the components of the exceptional divisors of the induced resolutions of the fibers) remain constant. The invariance of these intersection numbers is not proved in \cite{A}, but we hope this will be the case, either directly  from Artin's construction or from a slightly modified one. We indicate a possible proof.

In Section \ref{S:T}, we study the variation of the quantity $N_n$ along an Artin family. Such a family naturally induces a triple $(\sX,s,U)$, and we may apply to this one the construction of Section \ref{S:C} to obtain a Nash space of arcs $\sX_{\infty}^{(s)}$. It seems natural to use it to prove the constancy of $N_n$ along the Artin family. 

A possible way is to show that $\sX_{\infty}^{(s)}$ has finitely many irreducible components, say 
$\cA=\{\cA_1, \ldots, \cA_a    \}$, and that they induce the irreducible components of the arc spaces of the geometric generic and special fibers respectively. The finiteness of the set $\cA$ would follow if the triple  associated to our Artin family satisfies the mentioned Condition ($NO$). We prove that this condition is valid if $\La$ is equicharacteristic of characteristic zero, which unfortunately is not interesting for our applications. We make some comments about the general case, although so far we cannot prove it. 

In this section  we  also discuss another possible property of Artin families that, if valid, would indicate a form of ``equisingularity''. If one is able to construct such a family, satisfying this property, a key inequality to implement the present program would follow.

The article concludes by recalling the proof we propose to show the bijectivity of the Nash mapping for surfaces in positive characteristic, and enumerating the statements that should be verified to complete such a proof.

\section{}
\label{S:C}

\begin{voi}

\label{C1}
Here we address the construction of the arc scheme $\sX_{\infty}^{(s)}$ mentioned in the Introduction. 

We do not strive for maximum generality, but limit ourselves to a situation useful for our applications. Our basic set-up will be as follows.

{\bf Nice triples.} {\it Assume we have}:

\begin{itemize}
\item  Noetherian rings $\La$ and $R$ and  a ring homomorphism $\phi : \La \to R$,   which makes $R$ into a  $\La$-algebra.
\item An ideal $I$ of $R$ such that $R$ is $I$-complete and   the composition $\La \stackrel{\phi}{\to} R \to R/I$ is an isomorphism.
\end{itemize}

\smallskip
 
 Often we'll omit the algebra structure homomorphism $\phi$ and just consider the triple $(R,I,\La)$ (where $R$ is a $\La$-algebra and   $\La=R/I$). Such a triple satisfying the conditions just described, will be called a {\it nice triple}.

\end{voi}

\begin{voi}
\label{C2}
 In \ref{C1}, let $U=\Spec (\La)$ and $\sX = \Spec (R)$. Then, the homomorphism $\phi$ induces a projection $p: \sX \to U$ and the ideal $I$ defines a section $s:U \to \sX$ of $p$. Let  $S=V(I)$ be the image of $s$. Then $S$ and $U$ are isomorphic.

A triple $(\sX,s,U)$ of this type (i.e., one coming from a nice triple) will be called a {\it nice affine geometric triple}, or simply a {\it nice geometric triple}.

If $S=V(I)$ is the image of $s$ and $u$ is point of $U$, then the fiber $X_u :=p^{-1}(u)$ is a scheme over $k(u)$ (the residue field). Then $s(u) \in S \subset X_u$ and $s(u)$ is a rational point of $X_u$, i.e., 
$k(s(u))=k(u)$. 

Later, the case where $\La$ is a complete discrete valuation ring will be of particular importance. In this case $U$ and hence $S=s(U)$, are regular and one dimensional.  Note that $U$ has two points, $\bf o$ (closed) and $\gamma$ (the generic one, open). 

\end{voi}

\begin{voi}
\label{C3}

 Assume (in \ref{C1}) that the ideal $I$ is generated by elements $y_1, \ldots, y_n$ of the ring $R$. According to \cite{E}, Theorem 7.16, there is a unique 
 homomorphism of $\La$-algebras from $\La [[{\bf Y}]]=\La[[Y_1, \ldots, Y_n]]$ (power series in $n$ variables) to $R$ sending $Y_i$ into $y_i$, which is surjective. Let $(f_1, \ldots, f_s)\La [[\bf {Y}]]$  be the kernel.

Now let $A$ be a $\La$-algebra, $A[t]$ the ring of polynomials in $t$, $m \ge 0$ an integer, and $A[t]^{(m)}:=A[t]/(t^{m+1})$. 
 Let $\bar t$ be the image of $t$ in $A[t]^{(m)}$. 

The $A$-module  $A[t]/(t^{m+1})$ has a free basis  
$1,\bar t, \ldots, {\bar t}^m$. Hence we may write any element of $A[t]^{(m)}$ as 
$ a_{0} + a_{1} {\bar t} + \cdots + a_{m} {\bar t}^m$, for suitable unique elements $a_{j} \in A$. Let 
$$\cP(A[t]^{(m)}) := \{  \alpha \in A[t]^{(m)}: \alpha =   a_{1} {\bar t} + \cdots + a_{m} {\bar t}^m  , \, a_{j} \in A , \, \forall j\} = ({\bar t})A[t]^{(m)}$$
(an ideal of 
$A[t]^{(m)}$). 

With the notation just introduced, we have:
\end{voi}

\begin{pro}
 \label{P:C4}
If elements $z_1, \ldots, z_n$ of $\cP(A[t]^{(m)})$ are chosen, there is a unique homomorphism of $\La$-algebras 
$\La[[{\bf Y}]] \to A[t]^{(m)}$ sending $Y_j$ to $z_j$. 
\end{pro}
\begin{proof}
 Notice that if $M={Y_1^{i_1}  \cdots Y_n^{i_n}}$ is a monomial (in $\La[[{\bf Y}]]$) where $m'=i_1+\cdots+i_n > m$, then  
${z_1^{i_1}  \cdots z_n^{i_n}}=0 $, since this element can be written as ${\bar t}^{m'}\beta$, where $\beta \in A[t]^{(m)}$. In view of this observation, given a series    $f(Y_1, \ldots, Y_n) \in \La[[\bf{Y}]]$, it makes sense to write  $f(z_1, \ldots, z_n)$ because all terms of order $>m$ vanish. Then 
 the claimed homomorphism is obtained by sending the series $f(Y_1, \ldots, Y_n) \in \La[[\bf{Y}]]$ into $f(z_1, \ldots, z_n)$. 
\end{proof}

\begin{voi}
\label{C5} 
Now consider the ring $R=\La[[{\bf Y}]]/(f_1, \ldots, f_s)$ (see \ref{C3}). Given elements $z_1, \ldots, z_n$ of $\cP(A[t]^{(m)})$ as above, there is a homomorphism of $\La$-algebras $R \to A[t]^{(m)}$ sending $y_i \in I \subset R$ to $z_i$ if and only if, in  $A[t]^{(m)}$                          
$$(1) \qquad f_j(z_1, \ldots, z_n) =0 \, ,j=1, \ldots, s \, .$$

We may write 
$z_i =  a_{i1} {\bar t} + \cdots + a_{im} {\bar t}^m$, for suitable, unique elements $a_{ij} \in A$, $i=1, \ldots, n, \, j=1, \ldots, m$. Substituting into (1), for each $j$ we get, in 
$A[t]/(t^{m+1})=A[t]^{(m)} $, 
$$f_j(a_{11} {\bar t} + \ldots + a_{1m} {\bar t}^m, \cdots,  a_{n1} {\bar t} + \cdots + a_{nm} {\bar t}^m )=0 \, ,           $$
The expression of the left hand side  may be expanded as 
$$  F_{j1}(a_{11}, \ldots, a_{nm}){\bar t} +\cdots  +        F_{jm}(a_{11}, \ldots, a_{nm}){\bar t}^m \, ,$$
for suitable polynomials $F_{jq}$, 
 $j=1, \ldots, s \, , q=1, \ldots, m$
 in the ring of formal polynomials 
$$\La[{\bf A}]:=\La [ A_{11}, \ldots, A_{nm}]$$ 
($mn$ variables).
The condition (1) 
is equivalent to 
$$F_{jq} ( a_{11}, \ldots, a_{nm})=0, ~ j=1, \ldots, s, \, , q=0, \ldots m \, .$$

Let us demote by $R_m$ (or $\cR_m(R,I,\La)$) the $\La$-algebra $\La [[{\bf A}]]/(F_{11}, \ldots, F_{sq})$.

\end{voi}

\begin{voi}
\label{C6} 
From the construction of the $\La$-algebra $R_m$, it follows that it represents a certain functor. 

Precisely, with $(R,I,\La)$ a nice triple as above letting $\cA _{\La}$ denote the category of $\La$-algebras, consider the functor $F_{(R,I)}^{(m)}: \cA _{\La} \to (Sets)$ such that if $A \in \cA _{\La}$, 
$$(1) \quad  F_{(R,I)}^{(m)}(A)= Hom^{(I)}(R, A[t]^{(m)})   $$
where the right hand side denotes the set of homomorphisms of $\La$-algebras sending $I$ into $\cP(A[t]^{(m)})$. Then from \ref{C5} we have a bijection (functorial in $A$)

$$(2) \quad F_{(R,I)}^{(m)}(A)= Hom^{(I)}(R, A[t]^{(m)}) =                Hom({\cR}_m (R,I,\La),A)     $$
with $Hom$ denoting the set of homomorphisms of $\La$-algebras. Thus, $R_m$ represents the functor $F_{(R,I)}^{(m)}$.

Note that although in the construction of $R_m$ we use a presentation $R=\La[[{\bf Y}]]/(f_1, \ldots, f_s) \, ,$ 
since $R_m$ represents a functor depending on $R$ and $I$ only, the choice of the generators $f_1, \ldots, f_s$ is irrelevant.
\end{voi}

\begin{voi}
\label{C6.1}
If $m'\ge m$,  the functorial isomorphisms (2) induce a homomorphism of $\La$-algebras $\sX_m^{(s)} \to  \sX_{m'}^{(s)}$. Passing to the direct limit, we get a $\La$-algebra
$$R_{\infty}=\cR_{\infty}(R,I,\La) : = \lim _{\rightarrow} \cR_{m}(R,I,\La) \, .$$
From the isomorphisms (2) of \ref{C6} it follows that for any $\La$-algebra $A$ we have
$$  (1)  \qquad Hom^{(I)}(R, A[[t]]) =                Hom({R}_{\infty},A) $$
where $R_{\infty} =\cR_{\infty}(R,I,\La)$ and $Hom^{(I)}$ denotes the set of homomorphisms of $\La$-algebras sending $I$ into $(t)A[[t]]$. That is, $R_{\infty}$ represents the functor 
$$F_{\infty}(R,I):\cA \to (Sets)$$
such that for a $\La$-algebra $A$, $F_{\infty}(R,I)(A)=Hom^{(I)}(R,A[[t]])$.

\end{voi}

\begin{voi}
\label{CL}

More geometrically, taking $\Spec$, in the notation of \ref{C2} we have the geometric nice triple 
$$(1) \qquad (\sX,U,s)$$
induced by $(R,I,\La)$.  The $\La$-algebra $R_m$ induces a scheme $\sX_m^{(s)}$ of finite type over $U$, with the property that for any $\La$-algebra $A$, we have a natural bijection 
$$(2) \quad     Hom ^{(S)}(\Spec(A[t]^{(m)}), \sX) = Hom(\Spec(A), \sX_m^{(s)})  \, ,        $$
where $Hom$ denotes morphisms of $U$-schemes and $Hom^{(S)}$ those morphisms sending the subscheme of $\Spec(A[t]^{(m)})$ defined by $(t)$ to $S$.

 If $m' \ge m$,  
we have an induced (affine) ``projection'' morphism  $\sX_{m'}^{(s)} \to \sX_m^{(s)}$. Thus we get a projective system of schemes (involving affine projections), passing to the limit we get an $U$-scheme  
$\sX_{\infty}^{(s)}= \Spec (R_{\infty})$, no longer of finite type. 

This scheme $\sX_{\infty}^{(s)}$ has the following property:
$$  (3) \quad Hom^{(S)} (\Spec (A[[t]]), \sX) = Hom(\Spec(A), \sX_{\infty}^{(S)})$$
where $Hom$ denotes morphisms of $U$-schemes, and $Hom^{(S)}$ those morphisms sending the subscheme of $\Spec(A[[t]])$ defined by $(t)$ to $S$.

If $K$ is a field, a morphism $\Spec (K[[t]]) \to \sX$ sending the closed point of $\Spec (K[[t]]$ to a point of $S \subset \sX$ will be called an $S$-arc. 

\end{voi}

\begin{voi}
\label{Cal} 
Let $k$ be a field and $\cC _k$ denote the class of  $k$-algebras which are a quotient of a formal power series ring $k[[x_1, \ldots, x_n]]$ with coefficients in $k$. In other words, by Cohen's Structure Theorem (see \cite{MA}, Section 29), members of $\cC _k$ are complete noetherian $k$-algebras with residue field $k$.            An {\it algebroid surface} (over $k$) is a scheme of the 
 the form $X= \Spec (R)$ with $R$ a  two-dimensional ring in $\cC _k$. The closed point of such a scheme will be called the {\it origin}. 

Assume $R \in \cC_k$, with maximal ideal $M$. Then 
$(R,M,k)$ is naturally a nice triple, or $(X, s, \Spec (k))$ (where $X=\sX=\Spec(R)$ and $s$ is induced by the quotient map $R/M \to k$) is a nice  geometric  triple. If $O$ is the closed point of $X$ let us write 
$X_m^{(O)}:=\sX_m^{(s)}$ and $X_{\infty}^{(O)}:=\sX_{\infty}^{(s)}$.

 Thus we have constructed analogs of the spaces of $m$-arcs or of arcs (
 to a point), discussed in the Introduction  
 for an algebroid scheme $X$ whose closed point is a rational one, even if the base field $k$ is not algebraically closed.

\end{voi}

\begin{voi}
\label{C7}
The  algebra $R_m$ has a useful ``change of base ring'' property, that we explain next. 

Let $\psi: \La \to \La'$ be a homomorphism of noetherian rings. Given a nice triple  $(R,I,\La)$, consider the new triple $(R',I',\La ')$ obtained as follows. Let  
$I_{\La}=I(R \otimes _{\La} {\La '})$, $R'$ be the $I_{\La}$-completion of $R \otimes _{\La} \La'$, and $I'={\psi (I)} R'$. Then $( R', I',\La')$ is a nice triple, i.e., it satisfies the conditions of the set-up of \ref{C1}. 

Note that if $R$ admits a presentation 
$\pi:\La[[Y_1, \ldots, Y_n]] \to R$ with $\pi$ surjective and kernel $(f_1, \ldots, f_s)$, then $R'$ admits a presentation 
$\pi':\La '[[Y_1, \ldots, Y_n]] \to R'$,  where if $\pi(Y_i)=y_i \in R$, then  $\pi'(Y_i)=y'_i$ for all $i$, where $y_i'$ is the image of $y_i$ via the natural homomorphism $R \to R'$ (composition of $R \to R \otimes _{\La}
\La'$  and the mentioned completion). Moreover, if $f'_i$ is the image of $f_i$ in $\La'[[Y_1, \ldots, Y_n]]$, the kernel of $\pi '$ is the ideal $(f'_1, \ldots, f'_s)$. From this observation and the construction of \ref{C5}, it is easily checked that 
$\cR (R,I,\La) \otimes _ {\La}    {\La '} $ represents the functor $F^{(m)}_{(R',I')}$. That is, 
$\cR (R',I',\La')=   \cR (R,I,\La) \otimes _ {\La}    {\La '}$ (ordinary tensor product).

Notice that instead, $R'$ is the completed tensor product of $R$ and $\La'$ over $\La$, if in $R$ we consider the $I$-adic topology and both in $\La$ and $\La'$ the discrete ones.

\end{voi}

\begin{voi}
\label{C9}
In particular, given a nice triple $(R,I,\La)$ as in \ref{C1}, if $P$ is a prime ideal of $\La$ and $k(P)$ is the residue field of ${\La}_P$, we may take as our homomorphism  $\La \to \La '$  the  composition of the localization map $\La \to \La_P$ and the quotient map $\La_P \to \La_P/max(\La_P)=k(P)$. As in \ref{C7}, we get an induced nice triple $(R',I',k(P))$. We also get, as in \ref{C5} and \ref{C6.1}, the rings 
$$ (1) \quad  \cR_{m}(R',I',k(P)) ~ {\mathrm {and}} ~    \cR_{\infty}(R',I',k(P))$$
corresponding to truncated arcs and arcs respectively.

More geometrically, as in \ref{C2} we have the geometric nice triple $(\sX,s, U)$ induced by $(R,I,\La)$.  If $u$ is the point of  $U=\Spec (\La)$ corresponding to the prime ideal $P$,  we obtain (from 
$(R',I',k(P))$) a nice geometric triple $(\sX' _u, s', \Spec (k(P)))$. We also get, taking Spec in (1) above, induced $k(P)$-schemes 
$\sX_m ^{s(u)}$ and $\sX _{\infty} ^{s(u)}$.

If in our geometric triple, the projection is $p:\sX \to U$, $u \in U$ (corresponding to the prime ideal $P \in \La $), and $X_u=p^{-1}(u)$, we get an induced triple $(X_u,s(u),k(u))$. However, this is not necessarily nice: the fiber $X_u$ is the spectrum of $R \otimes _{\La} k(u)$, which in general is not a complete ring. To get a nice triple we must proceed as above, i.e., consider the ring $R'$ described in \ref{C7}.
\end{voi}

\begin{voi}
\label{C10}

If $s(u)=x \in X_u$,  
 ${\widehat{\cO}}_{X_u,x}$ is the usual completion of $\cO_{X_u,x}$ with respect to its maximal ideal,  
$\widehat X _u = \Spec ({\widehat{\cO}}_{X_u,x})$, and 
$O \in \widehat X _u $ is its closed point, then there is an identification  
$${(\sX _m ^{(s)})}_u= {(\widehat X _u)}_m^{(O)}$$
(notation of \ref{C9}). Similar considerations apply if $m$ is substituted by $\infty$, or if we consider the geometric fiber at $u$ instead.

\end{voi}

\begin{voi}
\label{C8} 
If $\alpha$ is a point of $\sX^{(s)}_{\infty}$ and $K=k(\alpha)$ is its residue field, then there is a natural morphism 
$\Spec(K) \to \sX_{\infty}^{(s)}$ whose image is $\alpha$. By (2), this corresponds to an arc 
$D_K \to \sX$, where $D_K=\Spec(K[[t]])$, sending the closed point of $D_K$ to a point in $S$.

Conversely, an element of 
$Hom^{(S)} (\Spec (L[[t]],\sX)$ with $L$ a field which is a $\Lambda$-algebra, determines, by (3) of \ref{CL}, a point of $\sX_{\infty}^{(S)}$, namely the image of 
$\Spec(L)$ via the corresponding morphism $Hom(\Spec(L), \sX_{\infty}^{(S)})$.
\end{voi}

\begin{voi}
 \label{C11}
When we work with a variety $X$ over an algebraically closed field, an important result says that $X_{\infty}^{(S)}$, the scheme of arcs relative to a closed set $S \subset X$, has finitely many irreducible components. We shall see that the same conclusion holds for the space $\sX_{\infty}^{(s)}$ (corresponding to a nice triple $(\sX,U,s)$), at least if 
$U$ is the spectrum of a discrete valuation ring $\La$ and the following condition is satisfied. 

{\it Condition (NO)}. With the notation of \ref{C2}, let us say that a nice geometric triple $(\sX,U,s)$ with $U=\{\gamma, \bf o \}$ the spectrum of a discrete valuation ring, ($\gamma$ being the generic point), satisfies {\it condition ($NO$)} if for every irreducible component $\cA$ of  $ \sX _{\infty} ^{(s)}$ we have:
 $$\cA \cap (\sX _{\infty}^{(s)})_{\gamma} \not= \emptyset.$$

The following question ensues: {\it is Condition ($NO$) always valid?}

We do not know the answer, but some partial results are discussed in the next section.

\end{voi}

\begin{pro}
 \label{P:irr}
Consider a nice geometric triple $(\sX,s, U)$ (where $U$ is the spectrum of a discrete valuation ring $\La$) that satisfies Condition ($NO$). Let $\gamma$ be the generic point of $U$, $\{\cA_i\}$ (with $i$ in a suitable index set $I$) the set of irreducible components of 
 $ \sX _{\infty} ^{(s)}$ and $\cA^{\gamma}_i:= \cA_i \cap (\sX _{\infty}^{(s)})_{\gamma}$, for all $i$. Then  $\{A^{(\gamma)}_i\}_{i \in I}$ is the set of irreducible components $\cA$ of the fiber $(\sX _{\infty}^{(s)})_{\gamma}$.
\end{pro}

\begin{proof}
Note
 that we have a commutative diagram 
\begin{displaymath}
\begin{array}{ccccc}

{(X_{\gamma})}_{\infty}^{(Q)}  &  \stackrel{\alpha_1}{\rightarrow}&    {\sX}_{\infty}^{(s)}  &  {\rightarrow} &   \sX _{\infty}    \\

{\downarrow}& &{\downarrow}& & {\downarrow}   \\

{\{Q\}}  &  \stackrel{\alpha_2}{\rightarrow}&    S  &  {\rightarrow} &   \sX    \\

{\downarrow}& & &{\searrow} &   \downarrow \\

{\{\gamma\}}  &  &  \stackrel{\alpha_3}\longrightarrow &   &   U    \\

\end{array}
\end{displaymath}
where $Q=s(\gamma)$, $S=Im(s)$, and          $\alpha_1$, $\alpha_2$ and $\alpha_3$ are open immersions. Now, from 
$$\sX_{\infty}^{(s)} =  \bigcup _{i \in I} \cA_i         $$
(because the $\cA_i$ are the irreducible components of $\sX_{\infty}^{(s)}$), by intersecting with ${(X_{\gamma})}_{\infty}^{(Q)}={(\sX_{\infty}^{(S)})}_{\gamma})$ we get 
${(X_{\gamma})}_{\infty}^{(Q)} = \bigcup _{i \in I}{\cA}^{\gamma} _i$. To show that the ${\cA}^{\gamma}_j$ are the irreducible components we have to show that:  

(i) Each 
${\cA}^{\gamma}_j$ is irreducible, and 

(ii) We do not have any inclusion
$ {\cA}^{\gamma}_i \subseteq {\cA}^{\gamma}_j, ~ j \not= i$. 

To show (i), note that $\cA_j$ is irreducible in $\sX_{\infty}^{(s)}$. Now, $\cA^{\gamma}_j=\cA_j \cap {(X_{\gamma})}_{\infty}^{(Q)}$ is an open in $\cA_j$. Since an open in an irreducible space is irreducible,  
$\cA_j$ is irreducible.

Now, to prove (ii), assume $ {\cA}^{\gamma}_i \subseteq {\cA}^{\gamma}_j$ for some pair $ i \not= j$.
 Then taking closures in ${\sX}_{\infty}^{(s)}$, and noticing that  
$ {\cA}^{\gamma}_j$ is dense in $\cA_j$ for all $j$, because Condition ($NO$) holds.
 we get 
${\cA}_i \subseteq  {\cA}_j$, which is not possible since 
$\sX_{\infty}^{(s)}= \bigcup _{i \in I }\cA_i $ is the decomposition of $\sX_{\infty}^{(s)}$ as the union of irreducible components. 
\end{proof}

\begin{cor}
\label{C:fin}
In  \ref{P:irr}, assume in addition that the characteristic of the field $k(\gamma)$ is zero. Then, $\sX_{\infty}^{(s)}$ has finitely many irreducible components.
\end{cor}

\begin{proof}
 Let $k=k(\gamma)$, $k'$ be the algebraic closure of $k$ and $(\sX _{\infty}^{(s)})_{\gamma '}$ be the geometric fiber at $\gamma$. Then by \ref{C7}, 
the geometric fiber $(\sX _{\infty}^{(s)})_{\gamma '}$ is the pull-back of the fiber $(\sX _{\infty}^{(s)})_{\gamma}$ via the natural morphism 
$\Spec (k') \to \Spec (k)$. But it is known that, 
denoting by $Ir(Z$) the set of irreducible components of a scheme $Z$, there is a surjective mapping 
$Ir((\sX _{\infty}^{(s)})_{\gamma '})    \to Ir((\sX _{\infty}^{(s)})_{\gamma})$. So, to prove the corollary it suffices to show that 
$Ir ((\sX _{\infty}^{(s)})_{\gamma '}) $ is finite. But, if $x' = s(\gamma ') \in X_{\gamma '}$,  $X' = \Spec (\widehat {\cO}_{X_{\gamma ', x'}})$ and $Q'$ is the closed point of $X'$, then by \ref{C10} there is an identification 
 $(\sX _{\infty}^{(s)})_{\gamma'}= (X')_{\infty}^{(Q')}$. The finitude of $Ir ((X')_{\infty}^{(Q')})$ is known (see \cite{IK}, Theorem 2.15).
\end{proof}

\begin{rem}
 \label{R:good}
In \ref{C:fin} we may drop the assumption on characteristic zero if we substitute ``irreducible components'' by ``good irreducible components'' (see the Introduction.)
\end{rem}

\section{}  
\label{S:S} 

In this Section we introduce a reasonable ``candidate'' to be the family of surfaces mentioned in 
 the Introduction. 
 For that, we shall use the results of Section \ref{S:C} together with a theorem of M. Artin that we shall review  next.  Before, we recall some notions.

\begin{voi}
 \label{Se:1}

Let $X$ be a normal surface defined over an algebraically closed field. A {\it good resolution} of the singularities of $X$ is a proper birational morphism 
$f:X_1 \to X$, with $X_1$ regular, which is an isomorphism off $\sg (X)$, such that the exceptional set $E(f)$ of $f$ consists of regular curves $E_i \subset X_1$ so that 
 (i) no point of $X_1$ belongs to three components of $E(f)$, and (ii) any two components of $E(f)$ meet transversally at every common point.
  
A good resolution  of $X$ is {\it very good}  if given two different components $E_i$ and $E_j$ of $Ex(f)$ then  $E_i \cap E_j$ is either empty or it has just one point.
\end{voi}

Next we state the mentioned theorem of Artin. In \ref{Cal} the class of rings $\cC _k$ was introduced.   
\begin{thm}
\label{T:A}
 Let $R \in \cC _k$ be a complete normal two dimensional domain, where  $k$ is an algebraically closed field of arbitrary  characteristic. Let $X=Spec (R)$ and 
$f: X_1 \to X $ be a  good resolution of singularities, with exceptional divisor having irreducible components $E_1, \ldots, E_m$.  
 Then, there exist:

\begin{itemize}

\item[(a)] a complete discrete valuation ring $\La$, with residue field $k$ and field of fractions $F$ of characteristic zero,

\item[(b)] a local complete normal $\Lambda$-algebra $\tilde R$ of dimension 3,

\item[(c)] a morphism $s: U \to \sX$ (where $U=\Spec (\La)$, and $\sX=\Spec(\tilde R)$)  that is a section of the natural map $\pi: \sX \to U$,

\item[(d)] a resolution of singularities $\tilde f : {\tilde \sX} \to \sX$,

\item[(e)] a local subring $R_0 \subset R$ such that $R$ is the normalization of $R_0$,
\end{itemize}
in such a way that these objects satisfy the following:

\begin{itemize}
\item[]
\begin{itemize}
\item[(i)] $\tilde f$ is an isomorphism off $Im(s)=S  \subset \sX$, and ${\tilde f}^{-1}(S)$ is the union of surfaces  
 ${\tilde E}_i$, smooth over $U$, with normal crossings.
\item[(ii)] On the closed fiber of $\tilde \pi : {\tilde \sX} \to U$ (where $\tilde \pi = \pi \tilde f)$, $\tilde \sX$ and ${\tilde E}_i$ induce $X_1$ and $E_i$, $i=1, \ldots, m$, respectively,
\item[(iii)] The closed fiber $X_0$ of $\pi$ is of the form $X_0=\Spec (R_0)$. 
\end{itemize}
\end{itemize}


\end{thm}

In other words, the theorem asserts the existence a family $\pi: \sX \to U$ having certain properties. A family satisfying the conditions and  conclusions of Theorem \ref{T:A} will be called an \it Artin family (of surface singularities).

\begin{rem}
\label{S:R1}
In the notation of the theorem, we have:

 (a) The scheme  $X_0$ is necessarily integral and its origin
$P_0$ is an {\it unibranch} point of $X_0$, i.e., the integral closure of ${\cO}_{X_0,P_0}$ is a local ring (\cite{G}, III (4.3.6)). In particular, the normalization morphism $g:X \to X_0$ (induced by the inclusion $R_0 \subset R$) is a homeomorphism, thus the origin of $X$ is the only point lying over the origin of $X_0$.

(b) Let $f_0=gf$, then $f_0:X_1 \to X_0$ is a good resolution of $X_0$. More generally, if $f':X' \to X$ is a resolution of $X$, then the composition 
$gf':X' \to X_0$ is a resolution of $X_0$. If $f'$  is a minimal resolution of $X$, then $gf'$ is a minimal resolution of $X_0$. If $f'$ is  good,  $gf'$ is also good.

\end{rem}

\begin{voi}
\label{Se:2}
Some of the notions presented in 
 the Introduction 
can be discussed in a somewhat more general context. Namely, assume $X_0$ is a surface (algebraic or algebroid) over an algebraically closed field $k$ with a single unibranch singularity $P_0$. Then its normalization $\eta: X \to X_0$ is a homeomorphism, in particular there is a single point $P \in X$ corresponding to $P_0$. Note that 
if $f: X_1 \to X$ is the minimal resolution of $X$, the composition $\eta f$ is the minimal resolution of $X_0$. We may define 
$N_r(X_0,P_0):=N_r(X,P)$. As before, in general we simply write $N(X_0)$. 
\end{voi}

\begin{voi}
\label{Se:3}

 Now we shall study some connections between minimal resolutions and good  resolutions of a surface. 

Consider a surface $X_0$ over an algebraically closed field $k$ of arbitrary characteristic,  with a single unibranch singularity $P_0$. Then there is a minimal resolution $f':X' \to X_0$ (obtained, as in \ref{Se:2}, from a minimal resolution of the normalization of $X_0$).  The irreducible components 
 of $Ex(f')$ are curves, perhaps singular and meeting badly. Let us also fix a 
 good resolution$f_1:X_1 \to X_0$, where the components of $Ex(f_1)$ are smooth curves, and $Ex(f_1)$ is a normal crossings divisor. We have a proper birational morphism $g:X_1 \to X'$ 
such that there is a commutative triangle 
 
\begin{displaymath}
\begin{array}{ccc}

X_1 &  \stackrel{g}{\rightarrow}&   X'   \\

&{f_1}{\searrow} & {f' \downarrow} \\

 & & X_0     \\
\end{array}
\end{displaymath}

Ordering things suitably, we may assume that the components $E'_1, \ldots, E'_m $ of $Ex(f')$ and those of $Ex(f_1)$, 
$E_{11}, \ldots, E_{1m}, \ldots, E_{1m_1}$, satisfy the following condition: $E_{1j}$ is the strict transform of $E'_j$ via $g$, for $j=1, \ldots, m$, while 
$g(E_{1,j})$ is a point of $X'$, for $j > m$.

Then, with the notation just used, $E_{11}, \ldots, E_{1m}$ 
 are precisely the essential components of the exceptional divisor of $f_1$, i.e., those that appear in any resolution of $X_0$.

\end{voi}

\begin{voi}
\label{V:Se4}
Now let $X_1$ and $X_2$ be surfaces over algebraically closed fields $k_1$ and $k_2$ respectively, both with a single unibranch singularity.  Assume $f_i:Z_i \to X_i , \, i=1,2$ is a good resolution and that $Ex(f_i)$ has  $m$ irreducible components $E_{i1}, \ldots, E_{im}$, $i=1,2$, and $j=1, \ldots, m$. Suppose we are able to order our objects in such a way that via the correspondence $E_{1j} \to E_{2j}$, we have  
 $(E_{1j}.E_{1q}) = (E_{2j}.E_{2q})$ for all possible pairs $j,q$ (i.e., the respective intersection numbers matrices coincide).
  Then, $E_{1j}$ is essential for $f_1$ if and only if $E_{2j}$ is essential for $f_2$.

Hence, if $f_i:X_i' \to X_i$, is the    the minimal resolutions of $X_i$,  $i=1,2$, then the number of irreducible components of the exceptional divisor of $f_1$ is the same as the  number of irreducible components of the exceptional divisor of $f_2$.
\end{voi}

\begin{rem}
 \label{R:Ss}
We hope that in a suitable Artin family (Theorem \ref{T:A}, whose notation we follow), if $X_{\bar \gamma}$ is the geometric general fiber of $\tilde \pi$, we have
$$(1) \qquad N_r(X_0) = N_r(X_{\bar \gamma})\,.$$
This is a possible approach to a proof. 
 Letting in \ref{T:A} ${\tilde X}_{\bar \gamma} : = {\tilde \pi}^{-1}(\bar \gamma)$ and ${\tilde X}_0 $ respectively denote the geometric generic and special fibers of $\tilde \pi : {\tilde \sX} \to U$, notice that the morphisms 
${\tilde \pi}_{\bar \gamma}:{{\tilde X}}_{\bar \gamma} \to X_{\bar \gamma}$ and 
 ${\tilde \pi}_{0}:{{\tilde X}}_{0} \to X_{0}$ 
 induced by $\tilde \pi$ are good resolutions of $X_{\bar \gamma}$ and $X_0$ respectively. In the notation of \ref{T:A}, $X_{0}$ can be identified to $X_1$ of that theorem. We intend to apply \ref{V:Se4} to the resolutions 
${\tilde \pi}_{\bar \gamma}$
and
 ${\tilde \pi}_{0}$. To that purpose, note that according to \ref{T:A}, the components of the exceptional locus $Ex(\tilde \pi _0)$ can be identified with $E_1, \ldots, E_m$. Let the components of $Ex({\tilde \pi}_{\bar \gamma})$ be denoted by $E'_1, \ldots, E'_m$. To apply \ref{V:Se4} we need to know that (after reordering indices if necessary) we have equalities of intersection numbers 
 $(E_i.E_j) = (E'_i.E'_j)$, for all possible pairs $(i,j)$. If so, the equality (1) easily follows from \ref{V:Se4}.

The point is to define a total intersection number 
$({\tilde E}_i. {\tilde E}_j)$ 
for the components of the exceptional divisor of the resolution $\tilde f$ in Theorem \ref{T:A}. An elementary way of doing this is to follow, with minor changes, Chapter IV, Section 1 of \cite{Sh}.
 That is, if ${\tilde E}_i$ and ${\tilde E}_j$ meet properly, the intersection having as irreducible components curves $C_1, \ldots, C_r$, with generic points 
$P_1, \ldots, P_s$ respectively, let $\la _i:=length (\cO_{C_i, P_i})$ and define  
 $({\tilde E}_i.{\tilde E}_j)=\Sigma_{i=1}^{r} \la _i$. If the intersection is not proper (e.g., when $i=j$) substitute $\tilde E _j$ by a suitable linearly equivalent divisor that
 meets  ${\tilde E}_i$ properly, and follow the previously described prescription. 

Probably these numbers can be defined in a sleeker way using more advanced intersection theories, e.g., that of \cite{F}.

Once we have these intersection numbers, we should verify that if $E_i$ (resp. $\bar E _i$) is the intersection of ${\tilde E}_i$ with the special fiber (resp. geometric generic fiber) of $\tilde \pi$, then    
$$ (*) \quad ({\bar E}_i. {\bar E}_j) = 
 ({\tilde E}_i. {\tilde E}_j) =
 ({ E}_i. {E}_j)$$
for all pairs $(i,j)$. Note that ${ E}_i$ and ${\tilde E}_i$ are respectively the irreducible components of the exceptional divisors in the induced good resolutions of the special and geometric general fibers of $\tilde \pi$.

The first equality in $(*)$ is simple, the second looks more complicated. It would be an extra ``equisingularity'' condition satisfied by Artin families.

If we use Theorem \ref{T:A} starting with a {\it very good resolution} $f:X_1 \to X$, in case  $i \not= j$, the equality $(E_i.E_j)=(E'_i.E'_j)$ is more easily proved. Indeed, 
 although not stated in \cite{A}, the method of proof yields the following result. If the 
 $I=\{ (i,j): i \not=j ~and ~ E_i \cap E_j \not= \emptyset  \}$, then for each $(i,j) \in I$ there is a section $\sigma _{i,j}$ of the induced projection 
$\tilde \sX \to U$ such that $Im (\sigma_{i,j})=\cE_i \cap \cE_j$. Note that $E_i \cap E_j \not= \emptyset $ means that $E_i \cap E_j$ has just one point.

Perhaps in the proof $(*)$ in general it also will help to assume that $f$ in \ref{T:A} is a {\it very good resolution}.
\end{rem}

\section{}
\label{S:T} 
\label{V:T1}

\begin{voi}
\label{T1}  In  Section \ref{S:C}  we have studied how  to develop a theory of the  arc scheme $\sX _{\infty}^{(s)}$ when we work with a nice geometric triple, see \ref{C2}. Notice that under the assumptions of Theorem \ref{T:A} we have a nice geometric triple associated to an Artin family. 

Namely, with the notation of that theorem, the triple is $(\sX,s,U)$. The only part that perhaps is not obvious is the fact that (with $I$ the ideal of $\tilde R$ defining $Im(s) \subset \sX$), $\tilde R$ is $I$-complete. But it is known that a noetherian local ring, complete with respect to the maximal ideal, is $I$-complete for any ideal $I$ of $R$ (see \cite{MA}, p. 63). 

Consequently, in this Section we shall apply the results of Section 2 to this triple. 

\end{voi}

\begin{voi}
\label{T2.0}
 We hope that the triple $(\sX,s,U)$ associated to an Artin family just introduced satisfies condition ($NO$) of \ref{C11}. Note that in this case $U$ has two points, the generic one $\gamma$ and the closed one $\bf o$. Now the condition 
$\cA \cap (\sX _{\infty}^{(s)})_{\gamma} \not= \emptyset$ (in the notation of \ref{C11}) means:

{\it (C) no irreducible component of $(\sX _{\infty}^{(s)})_{\gamma}$ is contained in ${(\sX ^{(s)}_{\infty})}_{(\bf o)} = (X_0)^{s({\bf o})}_{\infty}$}

To prove assertion $(C)$ it seems convenient to use the language of {\it specializations}, that we review next. 
\end{voi}

\begin{voi}
\label{T2.1}
Consider a nice geometric triple $(\sX,s,U)$. For simplicity, we assume $U=\Spec (\La)$, where $\La$ is a discrete valuation ring. In the sequel, if $K$ is a field, ${\mathbb D}_K$ denotes $\Spec(K[[t]])$.

Let $\alpha: {\mathbb D}_K \to \sX$ be an $s$-arc. Suppose we have a morphism 
$${\tilde \alpha} : \Spec (K[[u,t]]) \to \sX$$
(with $u,t$ analytically independent over $K$, note $K[[u,t]]=K[[u]][[t]]$) such that for $u=0$ the induced morphism is $\alpha$. If $K \subseteq L$ is a field extension, there is an induced $s$-arc 
$${\tilde \alpha}_L: \Spec (L((u))[[t]]) \to \sX$$
We say that the arc $\alpha$ is a {\it direct specialization} of the arc ${\tilde \alpha}_L$.

Notice that if $\alpha$ is a direct specialization of ${\tilde \alpha}_L$, and (according to (2) in \ref{C8}), $P$ and $Q$  respectively are the corresponding points in  ${\sX}^{(s)}_{\infty}$, then $P$ is in the Zariski closure of $Q$. That is, $P$ is a specialization of $Q$ in the sense of the theory of schemes. 

One could introduce a concept of {\it specialization} as a concatenation of direct specializations. We do not go into the details since we are not going to use this more general notion. 
\end{voi}
 
\begin{voi}
\label{T2.2}
So, to prove $(C)$ it suffices to prove this statement:

{\it (D) Any arc $\phi : \Spec (K[[t]]) \to X_0 \subset \sX$ is a direct specialization of an arc $\psi : \Spec (L[[t]]) \to X_0 $ (for a suitable field $L \supseteq K$) such that 
$Im(\psi) \nsubseteq X_0$.}

We prove that $(D)$ is valid when the base field $k$ (in \ref{T:A}) has zero characteristic. We do not have a proof in the positive characteristic  case (the one interesting to us), but in \ref{Tne} we shall make some comments on the difficulties to extend the proof to this situation. 

For simplicity we prove $(D)$ if $K=k$ (the base field), the general case is similar. So, letting $\mathbb D$ denote $\Spec (k[[t]])$, consider an arc 
$$\phi: \D \to X_0 \subset \sX ~,$$
it must send the origin of $\D$ to $s(\bf o)$, the single closed point of $\sX$. Take the desingularization 
$\tilde f : {\tilde \sX} \to \sX$ of \ref{T:A} and 
lift $\phi$ to an arc $\tilde{\phi} : \D \to \tilde{\sX}$, note that  
 $Im(\tilde{\phi}) \subset {\tilde{\sX}}_{\bf o} =\tilde{X _0}$, the desingularization of $X_0$ induced by $f$ of \ref{T:A}. Again the image of the closed point of $\D$ is a closed point $Q$ of $\sX$. We 
 get an induced local homomorphism of local rings $R={\cO}_{\tilde{\sX},Q} \to k[[t]]$ and, taking completions, one 
$$\varphi: {\tilde R} \to k[[t]] ~, $$ 
where ${\tilde R}$ denotes the completion  of  ${\cO}_{{\tilde \sX},Q} $ with respect to its maximal ideal.

Now, from the conditions satisfied by the exceptional locus ${\tilde f}^{-1}(S)$ in Theorem \ref{T:A}, there are two possibilities for the point $Q$:

(a) $Q$ belongs to a single component, say $\cE _1$, of ${\tilde f}^{-1}(S)$, or

(b) $Q$ belongs to two components, say $\cE _1$ and $\cE _2$, of ${\tilde f}^{-1}(S)$. 

Since ${\tilde R}$ is a regular local complete 3-dimensional $k$-algebra, whose residue field is again $k$, we may write ${\tilde R}=k[[u,x_1,x_2]]$ and we may choose the parameters in such a way that:

($\alpha$) In case (a), $x_1$ defines $\cE _1$ at $Q$,

($\beta$)
In case (b), $x_i$ defines $\cE _i$ at $Q$, $i=1,2$, and the ``$u$-axis'' coincides with $\cE_1 \cap \cE_2$. 

Consider the homomorphism $\varphi$, let $\varphi(x_i)=\phi _i(t) \in k[[t]]$, $i=1,2$. Then the homomorphism 
$${\tilde \phi} : {\tilde R} = k[[u,x_1,x_2]] \to k[[u,t]]$$
sending $u$ to $ut$ and $x_i$ to $\phi_i (t)$, $i=1,2$ defines a morphism 
$${\tilde \psi} : \Spec(k[[u,t]]) \to \Spec({\tilde R})\, .$$
Composing with the morphism $\Spec({\tilde R}) \to \Spec(R)$ we see that $\phi$ is a direct specialization of the arc 
$\psi : \Spec (k((u))[[t]]) \to \sX$ induced by $\tilde \phi$, and clearly its image is not included in $X_0$.

This proof is inspired by that of Lemma 2.12 in \cite{IK}.

\end{voi}
\begin{voi}
\label{Tne}
Can the proof of \ref{T2.2} be adapted to show that an Artin family satisfies condition $(NO)$ in case the discrete valuation ring $\Lambda$ in Theorem \ref{T:A} is not equicharacteristic?

The reduction to proving statement ($D$) in \ref{T2.2} is general, including the discussion on specializations in \ref{T2.1}. Concerning ($D$), the given arguments again work in general, up to the point where it is claimed that the local ring $\tilde R$ is a power series ring with coefficients in the field $k$ or, letting $\Gamma = k[[u]]$ (a complete discrete valuation ring with uniformizing parameter $u$), that 
$\tilde R = \Gamma [[x_1,x_2]]$. The claim in \ref{T2.2} is valid by applying Cohen's Structure Theorem in the equicharacteristic case. If $\Lambda$ (in \ref{T:A}) is a complete local discrete valuation ring with field of fractions of characteristic zero and residue field of characteristic $p > 0$, then $\tilde R$ is no longer equicharacteristic. If $M=max(\tilde R)$ and $k=\tilde R /M$ has characteristic $p > 0$, then Cohen's Theorem becomes more complicated. Namely, there are two cases (see \cite{MA}, Section 29):

(i) $p=p {1_{\tilde R}}$ is in $M$, 
but $p^2 \notin M$ (the {\it unramified} case.)

(ii) $p^2 \in M$ (the {\it ramified} case.)

In case (i), $\tilde R$ contains a {\it Cohen ring} (i.e., a complete discrete valuation ring $\Gamma$ with uniformizing parameter $p$, containing $\mathbb Z$), and $\tilde R \approx \Gamma [[x_1,x_2]]$ (a power series ring). 

In case (ii), $\tilde R$ contains an unramified complete regular local subring $R_1$ such that the extension $R_1 \subset {\tilde R}$ is {\it Eisenstein}. This means that $\tilde R$ is isomorphic to 
$R_1[T]/(h)$, where $h$ is a polynomial $T^m + a_{1}T^{m-1} + \cdots + a_m \in R_1[T]$, where $a_i \in max(R_1)$ for all $i$ but $a_m \notin max(R_1)^2$. 

In case (i) one is tempted to define a homomorphism $\Gamma \to \Gamma[[t]]$ sending $p$ to $pt$, and repeat the argument of \ref{T2.2}. But we do not see how to get such a homomorphism, or some alternative homomorphism 
$\Gamma[[x_1,x_2]] \to \Gamma[[t]]$ (for suitable elements $x_1,x_2$), allowing us to finish the proof as in \ref{T2.2} in this unramified case..

Case (ii) involves a further complication. If it were possible to find the desired specialization in case (i), then we could apply the result to $R_1$ but then we should reach a similar conclusion for its extension $\tilde R$.
 It does not seem clear how to accomplish this.  
\end{voi}
\begin{voi}
 \label{T2}
Summarizing, if $(\sX,s,U)$ is a nice triple associated to an Artin family satisfying condition $(NO)$ (we hope that always this will be the case) then by Proposition \ref{P:irr} and its Corollary, the arc space 
 $\sX_{\infty}^{(s)}$ 
has a finite number of irreducible components $\cA_1, \ldots, \cA_a$.
 Moreover, if the geometric generic fiber $(\sX_{\infty}^{(s)})_{\gamma '}$ has $a'$ irreducible components, then $a' \ge a$ (see  the proof of Corollary \ref{C:fin}). 

In other words,  if $X_{\gamma '}$ is the geometric general fiber of $\sX \to U$, then 
$N_n(X_{\gamma '}) \ge a$, and  a related interesting question may be posed: is $N_n(X_{\gamma'}) = a$?.
\end{voi}

\begin{voi}    
 \label{V:T8} To finish the proof of the surjectivity of the Nash map $\cN$ in characteristic $p$, we would like to have the following chain of inequalities:
$$ (1) \qquad N_r(X_0,P_0) \le N_r(X_{\gamma '},\bar Q)=  N_n(X_{\gamma '},\bar Q) \le N_n(X_0,P_0)$$
where we employ the ``usual'' notation, in particular $\gamma '$ denotes the geometric general point. 

Of the inequalities (1), the first one would be correct if the conclusion of Remark \ref{R:Ss} were valid, and we would get an equality. 
 
The second inequality is correct by the main result of \cite{FP}, moreover it is an equality.

The third one looks harder and to deal with it, next we present some comments on a possible proof. 
\end{voi}

\begin{voi}
 \label{E1} 

An affirmative answer to a question about Artin Families would imply the validity of the third inequality of the string (1) in \ref{V:T8}. The question is about a possible property of a family that would indicate some kind of equisingularity.

Consider an Artin family as in \ref{T:A}.   
 In particular, the family is parametrized by $U=\Spec (\La)= \{\bf o, \gamma \}$, $\La$ a non equicharacteristic complete discrete valuation ring, and $\gamma$ the generic point of $U$.

Let $U':=  \{\bf o, \gamma' \} $, where $\gamma'$ denotes the geometric generic point of $U$ (i.e., 
the natural morphism $\Spec(L) \to U$, where $L$ is the algebraic closure of the field $F=k(\gamma)$, that is of the fraction field of $\Lambda$). For any $u \in U'$, let $Y_u$ denote the fiber at $u$ of the morphism 
$\tilde \pi = \pi \tilde f: \sX _{\infty}^{(s)} \to U $ and $X_u$ the fiber at $u$ of $\pi:\sX \to U$ (notation of \ref{T:A}). If $W$ is a scheme,  $Ir(W)$ indicates the set of irreducible components of $W$ and 
$\cC(W) $ the set of closed subsets of $W$.

Then for any $u \in U'$ we have a mapping of sets 
$$F_u : Ir(\sX _{\infty}^{(s)}) \to \cC(Y_u) \, ,$$ 
where for an irreducible component $\cA$ of $\sX _{\infty}^{(s)}$, $F_u(\cA) = \cA \cap Y_u$. 

Our question is:

($E$) {\it Does $F_u$ induce a bijective mapping $\tilde {F_u} : Ir(\sX _{\infty}^{(s)}) \to Ir(Y_u)$ for all $u \in \bar U$?}

More explicitly, ($E$) requires to prove:

\begin{itemize}

\item[(1)] If $\cA$ is an irreducible component of $\sX _{\infty}^{(s)}$, then $F_u(\cA)$ is an irreducible component of $Y_u$ (so there is an induced map 
$\tilde {F_u} : Ir(\sX _{\infty}^{(s)}) \to Ir(Y_u)$).

\item[(2)] $\tilde {F_u}$ (or $F_u$) is injective.

\item[(3)] If $C$ is an irreducible component of $Y_u$, then there is an irreducible component $\cA$ of $\sX _{\infty}^{(s)}$ such that $C=\cA \cap Y_u $

\end{itemize}

 We know that 
 $Ir(Y_{\gamma '})$ is finite, 
 say of cardinality $a$. If $(E)$ has a positive answer, then by the bijections, both $Ir(Y_0)$ and $Ir(\sX _{\infty}^{(s)})$ are finite, with cardinality $a$. In other words, 
$$N_n(X_{\gamma '}) = N_n(X_0) \, ,$$
and thus the third inequality of (1) \ref{V:T8} is valid, and  actually it is an equality.
\end{voi}

\begin{voi}
 \label{E2}
Even if not every Artin family satisfies condition ($E$), one could ask: given $X$ as in \ref{T:A}, can we construct an Artin family satisfying condition ($E$)? 

\end{voi}

\begin{voi}
\label{E3}

This article has described an approach that might yield a proof of the bijectivity of the Nash map $\cN$ for surfaces in positive characteristic. Our approach views a surface of such nature as the special member of a one parameter family of surfaces, whose general member is defined over a field of characteristic zero. To address such vision, the approach introduces the notion and properties of {\it geometric triples} (\ref{C2}), and finds in a ``lifting theorem'' of M. Artin, which determines in a natural way a geometric triple, a reasonable candidate to be the family needed. Following the approach, some properties of the targeted family have been proven. However, other properties should be further addressed if the proof envisioned is to be complete. 

Specifically, the following would be needed:

\begin{itemize}

\item[(a)] to prove Equality (1) in Remark \ref{R:Ss}, for instance as  suggested introducing a suitable  Intersection Theory,

\item[(b)] to prove that triples coming from an Artin family satisfy Condition (NO) of \ref{C11}, see \ref{T2.0} and,

\item[(c)] to prove  the inequality $N_n(X_{\gamma}, \bar Q) \le N_n(X_0,P_0)$  in (1) of \ref{V:T8}, which would follow if one could answer affirmatively question ($E$) in \ref{E1}, see also \ref{E2}.

\end{itemize}

\end{voi}

\providecommand{\bysame}{\leavevmode\hbox to3em{\hrulefill}\thinspace}

\end{document}